\documentclass[12pt]{amsart}

\usepackage{amsmath}
\usepackage{amsthm}
\usepackage{amssymb}

\usepackage{tikz}
\usetikzlibrary{calc}

\usepackage{url}
\usepackage{xcolor}

\usepackage[shortlabels]{enumitem}

\usepackage[colorlinks=true,urlcolor=blue,linkcolor=blue,plainpages=false,pdfpagelabels]{hyperref}

\newtheorem{theorem}{Theorem}[section]
\newtheorem{definition}[theorem]{Definition}

\newtheorem{proposition}[theorem]{Proposition}

\newtheorem{lemma}[theorem]{Lemma}
\newtheorem{remark}[theorem]{Remark}
\newtheorem{example}[theorem]{Example}

\newcommand{\N}{{\mathbb N}}
\newcommand{\R}{{\mathbb R}}

\newcommand{\eps}{\varepsilon}

\newcommand{\tot}{\leftrightarrow}

\newcommand{\dist}{{\rm{dist}}}

\usepackage[left=3.0cm,%
right=3.0cm,%
top=2.5cm,%
bottom=3.5cm,%
headheight=12pt,%
a4paper]{geometry}%

\begin{document}
\title{On the finitary content of Dykstra's cyclic projections algorithm}

\author[Pedro Pinto]{Pedro Pinto}
\date{\today}
\maketitle
\vspace*{-5mm}
\begin{center}
	{\scriptsize Department of Mathematics, Technische Universit\"at Darmstadt,\\
		Schlossgartenstra\ss{}e 7, 64289 Darmstadt, Germany, \ \\ 
		E-mail: \protect\url{pinto@mathematik.tu-darmstadt.de}}
\end{center}

\begin{abstract}
	We study the asymptotic behaviour of the well-known Dykstra's algorithm through the lens of proof-theoretical techniques. We provide an elementary proof for the convergence of Dykstra's algorithm in which the standard argument is stripped to its central features and where the original compactness principles are circumvented, additionally providing highly uniform primitive recursive rates of metastability in a full general setting. Moreover, under an additional assumption, we are even able to obtain effective general rates of convergence. We argue that such additional condition is actually necessary for the existence of general uniform rates of convergence.
\end{abstract}

\section{Introduction}\label{intro}

Many problems in convex optimization can be stated in terms of finding a point in the intersection of a family of convex and closed sets, what is known as the convex feasibility problem:
\begin{equation}\tag{CFP}
	\text{find some point}\  x\in \bigcap_{j\in I}C_j
\end{equation}
assuming \textit{a priori} that $\bigcap_{j\in I} C_j\neq \emptyset$, i.e.\ the problem has a solution (is feasible). The study of such problems first appeared in connection with constrains defined by linear inequalities and where the feasibility set is the intersection of half-spaces. Since then the general problem has been the subject of much research due to its broad applicability in applied mathematics -- e.g in statistics, partial differential equations (Dirichlet problem over irregular regions), solving linear equations (Kaczmarz's method), image or signal restoration, and computed tomography. For further discussions we refer the reader to the surveys \cite{BauschkeBorwein(1996),Deutsch(1992)}.

One of the most successful and well-known techniques to iteratively approximate a solution to the CFP is von Neumann's method of alternating projections \eqref{MAP}. For a subspace $V$, let $P_V$ denote the orthogonal projection map onto $V$.
\begin{theorem}[von Neumann~\cite{vonNeumann(1950)}]
	Let $V_1, V_2$ be two closed vector subspaces of a Hilbert space $X$. Then, for any point $x_0\in X$ the iteration defined by
	\[
	x_{n+1}:=P_{V_1}P_{V_2}(x_n)
	\]
	converges strongly to $P_{V_1\cap V_2}(x_0)$.
\end{theorem}

The original proof by von Neumann doesn't generalize immediately to more than two subspaces. This was overcome by Halperin who extended the convergence result to a finite number of subspaces.
\begin{theorem}[Halperin~\cite{Halperin(1962)}]
	Let $V_1, \cdots, V_m$ be $m\geq 2$ closed vector subspaces of a Hilbert space. Then, for any point $x_0\in X$ the iteration defined by
	\begin{equation}\label{MAP}\tag{{\rm{MAP}}}
		x_{n+1}:=P_{V_1}\cdots P_{V_n}(x_n)
	\end{equation}
	converges strongly to $P_{\bigcap_{j=1}^{m} V_j}(x_0)$.
\end{theorem}

The convergence of \eqref{MAP} holds more generally when the sets $V_j$ are affine subspaces (i.e. translates of subspaces) provided that their intersection is nonempty. However, if the sets are just assumed to be closed and convex then the situation is more delicate. In 1965, Bregman establish weak convergence of von Neuman's method in the general setting.
\begin{theorem}[Bregman~\cite{Bregman(1965)}]
	Let $C_1, \cdots, C_m$ be $m\geq 2$ closed convex subsets of a Hilbert space such that $\bigcap_{j=1}^mC_j\neq \emptyset$. Then, \eqref{MAP} converges weakly to a point in the intersection.
\end{theorem}
Hundal~\cite{Hundal(2004)} gave a counterexample (where the sets consist of a closed hyperplane and a closed convex cone in $\ell^2(\N)$) for which the iteration indeed doesn't converge in norm. Moreover, there are very easy examples of convex sets where \eqref{MAP} converges strongly to a point in the intersection but such point is distinct from the projection of $x_0$. A different iterative scheme was proposed by Dykstra that does converge strongly to the projection of $x_0$ in the general setting of the convex feasibility problem.

Consider $C_1, \cdots, C_{m}$ to be $m \geq 2$ closed convex subsets of a Hilbert space with nonempty intersection. For $n\geq 1$, let $C_n$ denote the set $C_{j_n}$ where $j_n:=[n-1]+1$ with $[r]:\equiv r\!\mod m$, and let $P_n$ denote the metric projection onto $C_n$. For $x_0\in X$ an initial point, Dykstra's cyclic projections algorithm is defined recursively by the equations
\begin{equation}\label{dykstra-iter}\tag{{\rm{D}}}
\begin{cases}
	x_0\in X\\
	q_{-(m-1)}=\cdots=q_0=0
\end{cases}
\, \text{ and }\, \forall n\geq 1\,
\begin{cases}
	x_n:=P_{n}(x_{n-1}+q_{n-m})\\
	q_n:=x_{n-1}+q_{n-m}-x_n
\end{cases}
\end{equation}

In 1983, Dykstra~\cite{Dykstra(1983)} proved the strong convergence in the particular case when all the sets are closed convex cones of a finite dimensional Hilbert space. The result was later extended to the general setting by Boyle and Dykstra~\cite{BoyleDykstra(1986)}.
\begin{theorem}[Boyle-Dykstra~\cite{BoyleDykstra(1986)}]\label{main}
	Let $C_1, \cdots, C_{m}$ be $m \geq 2$ closed convex subsets of a Hilbert space such that $C:=\bigcap_{j=1}^{m}C_j\neq \emptyset$. Then, for any point $x_0\in X$ the iteration $(x_n)$ generated by \eqref{dykstra-iter} converges strongly to $P_C(x_0)$.
\end{theorem}

When the $C_i$'s are closed vector subspaces (or more generally, closed affine subspaces), the projection is a linear map and it is easy to see that the scheme \eqref{dykstra-iter} reduces to \eqref{MAP}.
It may be helpful to think of Dykstra's algorithm as operating in stages: it starts with some initial guess $x_0\in X$ and by setting auxiliary terms $q_{-(m-1)}, \cdots, q_{-1},q_{0}$ to zero. One is then able to compute $x_1, \cdots, x_m$. Using these points, we can now \emph{update} the values of the auxiliary terms, namely computing $q_1, \cdots, q_{m}$. After this, the process repeats and will approximate in norm the projection of $x_0$ onto the feasibility set.


Dykstra's cyclic projections algorithm is an attractive method for strongly approximating the intersection of closed convex sets in a Hilbert space as, while it converges to the feasible point closest to $x_0$, it only requires knowledge about the projections onto the individual convex sets $C_j$. However, contrary to \eqref{MAP} not much is known regarding its quantitative information. For the particular case when the sets $C_j$ are closed half-spaces, i.e. the intersection is a polyhedral subset of $X$, Deutsch and Hundal~\cite{DeutschHundal(1994)} obtained rates of convergence for \eqref{dykstra-iter} (see \cite{Deutsch(1995)} and references therein for further discussions on rates of convergence and their applications). Strikingly, in the case of $m=2$ the rate doesn't depend on the initial point but only on an upper bound on its distance to the intersection set.

In this paper, we analyse the asymptotic behaviour of Dykstra's algorithm and provide highly uniform quantitative information in the general setting of Theorem~\ref{main}. When analysing the convergence of the sequence $(x_n)$, one considers the equivalent Cauchy formulation
\begin{equation}\label{cauchy}\tag{$\dagger$}
\forall \eps >0\ \exists n\in \N\ \forall i,j\geq n \left( \|x_i-x_j\|\leq \eps \right).
\end{equation}
Effective rates of convergence, i.e. a computable function giving a witness to the value of $n$ in \eqref{cauchy}, are in general excluded. Proof-theoretical considerations (see \cite{Kohlenbach(2008)i}) guides us to the next best possible thing
\begin{equation}\label{metastability}\tag{$\ddagger$}
	\forall \eps >0\ \forall f:\N\to\N\ \exists n\in \N\ \forall i,j\in [n;n+f(n)] \left( \|x_i-x_j\|\leq \eps \right),
\end{equation}
where $[n;n+f(n)]=\{n, n+1, \cdots, n+f(n)\}$. The statement \eqref{metastability} is noneffectively equivalent to the Cauchy property of the sequence. This reformulation of \eqref{cauchy} is long known to logicians as Kreisel's no-counterexample interpretation and was popularized by Terence Tao under the name of \emph{metastability} \cite{Tao(2008)i,Tao(2008)ii}. In line with such terminology, a function that given $\eps >0$ and $f\in\N^{\N}$ outputs a bound on $n$ in \eqref{metastability} is called a rate of metastability. General logical metatheorems (e.g.\ \cite{Kohlenbach(2005)i,GerhardyKohlenbach(2008),Pischke(2023),EngraciaFerreira(2020)}) guarantee the existence of such rates provided the proof of the convergence statement can be formalized in certain formal systems\footnote{The class of such admissible proofs is (in practice) very large, encompassing in particular most proofs in classical analysis.}. 

We provide a quantitative analysis of the proof in \cite{BauschkeCombettes(2017)} which follows closely the arguments originally given by Boyle and Dykstra \cite{BoyleDykstra(1986)}. Through our quantitative analysis, it was possible to remove the compactness principles crucial in the original proof and, in this way, obtain rates of metastability in the general setting of the convex feasibility problem which are primitive recursive (in $f$ in the sense of Kleene). Moreover, our quantitative results are a true finitization of Theorem~\ref{main}, in the sense that the infinitary result is fully recovered. In this way, we provide an elementary proof for the convergence of Dykstra's algorithm. Mathematicians naturally prefer rates of convergence to rates of metastability. However, in full generality they are usually unavailable. Furthermore, when such rates are actually available, they are frequently sensitive to the parameters of the problem. In this case, one would expect a rate of convergence to depend on the specifics of the underlying space, of the convex sets and on the initial point on which the iteration is initiated. This should be compared with the uniformity exhibited by the rate of metastability obtain in Theorem~\ref{main_quant}: it only depends on the number of convex sets and on a bound to the distance of $x_0$ and the intersection set. For the general case augmented with a regularity assumption on the convex sets, we show that it is possible to obtain uniform rates of convergence. From this, we derive rates of convergence for the case of basic semi-algebraic convex sets in $\R^n$ (which in particular covers the polyhedral case in finite dimension). We furthermore argue that the setting with this regularity assumption is actually the only one where uniform rates of convergence are possible, in particularly encompassing the rate of convergence obtained by Hundal and Deutsch for the polyhedral case with $m=2$.

The finitary study in this paper is set in the context of the `\emph{proof mining}' program (\cite{Kohlenbach(2008)i}, see also the recent survey \cite{Kohlenbach(2019)i}) where proof-theoretical techniques are employed to analyze \textit{prima facie} noneffective mathematical proofs with the goal of obtaining additional information. With a broad range of applications, the proof mining program has been in particular very successful in the study of results in convex optimization. Its applications in the study of algorithms approximating fixed points of nonexpansive maps in convex sets (like Krasnoselski-Mann and Ishikawa in normed spaces  \cite{Kohlenbach(2001)i,Kohlenbach(2001)ii,Kohlenbach(2003)} and \cite{Leustean(2007)i,Leustean(2010)} in a geodesic setting, Browder's implicit scheme \cite{Kohlenbach(2011)i,FerreiraLeusteanPinto(2019),KohlenbachSipos(2021)}, Halpern type iterations \cite{Kohlenbach(2011)i,KohlenbachLeustean(2012)i,Koernlein(2015),FerreiraLeusteanPinto(2019),KohlenbachPinto(2022)}, etc.) has been very fruitful providing new quantitative information (either in the form of rates of convergence or rates of metastability) and frequent even qualitative improvements and generalizations (as for example shown in the recent work \cite{KohlenbachPinto(2022),Sipos(2022)ii,DinisPinto(2023)ii}). Similarly, proof mining has also been an important tool in the quantitative study of methods converging towards zeros of monotone (or accretive) operators (with the study of the well-known proximal point algorithm in \cite{KohlenbachLeusteanNicolae(2018),KohlenbachLopes-AcedoNicolae(2019),Kohlenbach(2021)}, an abstract version in a geodesic setting \cite{LeusteanNicolaeSipos(2018),Sipos(2022)i,Sipos(2022)ii}, and of several Halpern-type variants in \cite{Pinto(2021),LeusteanPinto(2021),Kohlenbach(2020),Kohlenbach(2022)}). A closer connection with the CFP can be found in its applications to splitting methods e.g.\ \cite{DinisPinto(2021)ii,DinisPinto(2023)i,DinisPinto(2023)ii,LeusteanPinto(2023)}, in particular in the previous studies of \eqref{MAP} in \cite{Kohlenbach(2016),KohlenbachLeusteanNicolae(2018),KohlenbachLopes-AcedoNicolae(2019)}. An important work was the study of Bauschke's solution to the zero displacement problem \cite{Bauschke(2003)} regarding the potentially
`inconsistent feasibility problem', carried out in \cite{Kohlenbach(2019)ii}, also in the context of the proof mining program (see \cite{KohlenbachPischke(2023)}).

All these previous work point to the usefulness of proof mining and of logic-based techniques in the study of the convex feasibility problem and connected methods. That being said, we remark that such perspective and techniques only operate in the background and that the central theorems are presented in a way which doesn't assume any particular logic knowledge of the reader. Nevertheless, we allow ourselves some simple logical remarks in the final section of the paper, which we think may be of particular interest to logicians and the more inquisitive mathematician.

The paper is organized as follows. In section 2, besides recalling useful terminology and known results, we establish some technical lemmas. Sections 3 and 4 contain the main contributions of the paper. In section 3, we begin with some initial considerations on Dykstra's algorithm and by proving some results on its asymptotic regularity. Afterwards, we obtain a rate of metastablity for the iteration in the general context. In section 4, we discuss rates of convergence under the assumption that a modulus of regularity is available. Section 5 is devoted to our final considerations.

\section{Preliminaries and Lemmas}\label{pre}

\subsection{Quantitative notions}

We begin by recalling terminology for some quantitative notions. Let $(x_n)$ be a Cauchy sequence in a normed space $(X, \|\cdot\|)$.
\begin{definition}
	We say that a function $\theta:(0,\infty)\to \N$ is a Cauchy rate if
	\[
	\forall \eps >0\ \forall i,j\geq \theta(\eps) \left( \|x_i-x_j\|\leq \eps \right).
	\]
\end{definition}
If the sequence $(x_n)$ converges to some point $x\in X$ (e.g.\ if the space is complete), then any Cauchy rate will also be a rate of convergence towards $x$, i.e.
\[
\forall \eps >0\ \forall i\geq \theta(\eps) \left( \|x_i-x\|\leq \eps \right).
\]
However, considerations from computability theory tell us that effective Cauchy rates are in general excluded, and one thus looks at a reformulation, although equivalent, computationally weaker.
\begin{definition}
	We call a function $\Theta:(0,\infty)\times \N^{\N}\to \N$ a rate of metastability if
	\[
	\forall \eps>0\ \forall f:\N\to\N\ \exists n\leq \Theta(\eps,f)\ \forall i,j\in [n;n+f(n)] \left(\|x_i-x_j\|\leq \eps\right),
	\]
	where $[n;n+f(n)]:=\{n, n+1,\cdots, n+f(n)\}$.
\end{definition}

The following result is folklore (see e.g.\ \cite{KohlenbachPinto(2022)}).
\begin{proposition}\label{rate-meta}
	A function $\theta:(0,\infty)\to\N$ is a Cauchy rate if and only if the function $\Theta:(\eps,f)\mapsto\theta(\eps)$ is a rate of metastability.
\end{proposition}

\subsection{Lemmas}

 Throughout, let $(X, \langle\cdot, \cdot\rangle,\|\cdot\|)$ be a (real) Hilbert space with inner product $\langle\cdot,\cdot\rangle$ and norm $\|\cdot\|$. We also use the notation $[a;b]:=[a,b]\cap\N$ and $\N^*:=\N\setminus\{0\}$. In this section we collect known results and prove some useful technical lemmas. Recall the following characterization of the metric projection, known as Kolmogorov's criterium. 
\begin{proposition}\label{Kolmogorov}
	Let $C$ be a nonempty, closed and convex subset of $X$. Then, every point $u\in X$ has a unique best approximation on $C$, which we denote by $P_C(u)$ and call the projection of $u$ onto $C$. Furthermore, $P_C(u)$ is the unique element of $C$ satisfying 
	\[
	\forall y\in C\ \left(\langle u-P_C(u), y-P_C(u)\rangle \leq 0\right).
	\]
\end{proposition}

We shall require a quantitative version of the proposition above. The first quantitative studies on the metric projection featured in \cite{Kohlenbach(2010)} and \cite{Kohlenbach(2011)i}. The version used here follows the formulation proven in \cite{FerreiraLeusteanPinto(2019),Pinto(2019)}. In \cite{DinisPinto(2023)ii}, it was used in an `$\eps/\delta$-formulation' for the general nonlinear setting of CAT(0) spaces. Let $\overline{B}_r(p)$ denote the closed ball of radius $r\geq 0$ centred at $p\in X$, i.e.\ $\overline{B}_r(p):=\{x\in X : \|x-p\|\leq r\}$.
\begin{proposition}\label{quant-projection}
	Given $u\in X$, let $b\in\N^*$ be such that $b\geq \|u-p\|$ for some point $p\in \bigcap_{j=1}^m C_j$. For any $\eps >0$ and function $\delta:(0,\infty)\to (0, \infty)$, there exists $\eta\geq \beta(b,\eps,\delta)$ and $x\in \overline{B}_b(p)$ such that $\bigwedge_{j=1}^m \|x-P_j(x)\|\leq \delta(\eta)$ and
	\[
	\forall y\in \overline{B}_b(p) \left( \bigwedge_{j=1}^m \|y-P_j(y)\|\leq \eta \to \langle u-x, y-x\rangle \leq \eps \right),
	\]
	where $\beta(b,\eps,\delta):=\dfrac{\varphi^2}{24b}$ with
	\[
	\varphi:=\min\left\{\tilde{\delta}^{(i)}(1) : i\leq \left\lceil\frac{4b^4}{\eps^2}\right\rceil\right\}\ \text{and}\
	\tilde{\delta}(\xi):=\min\left\{ \delta\left(\frac{\xi^2}{24b}\right), \frac{\xi^2}{24b} \right\},\ \text{for any}\ \xi>0.
	\]
\end{proposition}

The proof of this result is an easy modification of the one in \cite{DinisPinto(2023)ii} regarding common (almost-)fixed points of two nonexpansive maps to common (almost-)fixed points of $m$ projection maps. Indeed, it is clear that the result holds \emph{mutatis mutandis} for any finite number of nonexpansive maps (and so for metric projections in Hilbert spaces). In order to convince the reader, we nevertheless include a proof. First, we require the following two technical lemmas essentially due to Kohlenbach~\cite{Kohlenbach(2011)i}.
\begin{lemma}\label{Lem-tech-Koh1}
	Let $C$ be some convex bounded subset of a normed space and $D\in\N^*$ be a bound on the diameter of $C$. Consider a nonexpansive map $T:C\to C$. Then,
	\[
	\forall \eps >0\ \forall x_1, x_2\in C \left( \bigwedge_{i=1}^2\|x_i-T(x_i)\|\leq \frac{\eps^2}{12D}\to \forall t\in [0,1] \left(\|w_t-T(w_t)\|\leq \eps\right)\right),
	\]
	where $w_t:=(1-t)x+ty$.
\end{lemma}

\begin{lemma}\label{Lem-tech-Koh2}
	Let $X$ be a normed space and for each $x,y\in X$ and $t\in [0,1]$ write $w_t:=(1-t)x+ty$. Then, for any $u,x,y\in X$
	\[
	\forall \eps\in(0,b^2] \left( \forall t\in[0,1] \left(\|u-x\|^2\leq \|u-w_t\|^2+\frac{\eps^2}{D^2}\right) \to \langle u-x, y-x\rangle \leq \eps \right),
	\]
	where $D\in\N^*$ is such that $D\geq \|x-y\|$.
\end{lemma}

\begin{proof}[Proof of Proposition~\ref{quant-projection}]
	Let $\eps>0$ and a function $\delta:(0,\infty)\to (0, \infty)$ be given. We structure the argument in two central claims.\\
	\textbf{Claim 1.} There exist $\eta\geq \varphi$ and $x\in \overline{B}_b(p)$ such that $\bigwedge_{j=1}^m\|x-P_j(x)\|\leq \tilde{\delta}(\eta)$ and
	\[
	\forall y\in \overline{B}_b(p) \left( \bigwedge_{j=1}^m \|y-P_j(y)\|\leq \eta \to \|u-x\|^2\leq \|u-y\|^2 + \frac{\eps^2}{4b^2} \right).
	\]
	\textbf{Proof of Claim 1.} Assume towards a contradiction that for all $\eta\geq \varphi$ and $x\in\overline{B}_b(p)$ such that $\bigwedge_{j=1}^m \|x-P_j(x)\|\leq \tilde{\delta}(\eta)$,
	\begin{equation}\label{eq1-projection}
	\exists y\in\overline{B}_b(p) \left( \bigwedge_{j=1}^m\|y-P_j(y)\|\leq \eta\ \land\ \|u-y\|^2<\|u-x\|^2-\frac{\eps^2}{4b^2} \right).
	\end{equation}
	We then define a finite sequence $z_0, \cdots, z_r$ with $r:=\left\lceil\frac{4b^4}{\eps^2}\right\rceil$ as follows. Take $z_0$ to be $p$. In particular, $z_0\in \overline{B}_b(p)$ and $\|z_0-P_j(z_0)\|\leq \tilde{\delta}^{(r)}(1)$ for all $j\in[1;m]$. Consider that for $i\leq r-1$, we have $z_i\in\overline{B}_b(p)$ such that $\|z_i-P_j(z_i)\|\leq \tilde{\delta}^{(r-i)}(1)$ for all $j\in[1;m]$. Then, by \eqref{eq1-projection} there is some $y\in\overline{B}_b(p)$ such that
	\[
	\bigwedge_{j=1}^m\|y-P_j(y)\|\leq \tilde{\delta}^{(r-i-1)}(1)\ \text{and}\ \|u-y\|^2<\|u-z_i\|^2-\frac{\eps^2}{4b^2}.
	\]
	We define $z_{i+1}$ to be one such $y$. By construction $\|u-z_{i+1}\|^2<\|u-z_i\|^2-\eps/4b^2$ for all $i< r$, and so we obtain
	\[
	\|u-z_r\|^2<\|u-z_0\|^2-r\frac{\eps^2}{4b^2}\leq b^2-\frac{4b^4}{\eps^2}\frac{\eps^2}{4b^2}=0.
	\]
	which is a contradiction and concludes the proof of the claim.\hfill$\blacksquare$\\
	
	\noindent In the following, consider $\eta_0\geq \varphi$ and $x\in \overline{B}_b(p)$ as per Claim 1, and define $\eta_1:=\frac{\eta_0^2}{24b}$ which is bounded below by $\beta(b,\eps,\delta)$.\\
	\textbf{Claim 2.} For all $y\in\overline{B}_b(p)$ and $t\in [0,1]$,
	\[
	\bigwedge_{j=1}^m\|y-P_j(y)\|\leq \eta_1 \to \|u-x\|^2\leq \|u-w_t\|^2+\frac{\eps^2}{4b^2}.
	\]
	\textbf{Proof of Claim 2.} By the definition of the function $\tilde{\delta}$, we have in particular
	\[
	\bigwedge_{j=1}^m\|x-P_j(x)\|\leq \eta_1.
	\]
	Now, if we consider $y\in \overline{B}_b(p)$ such that $\|y-P_j(y)\|\leq \eta_1$ for all $j\in[1;m]$, then we can apply Lemma~\ref{Lem-tech-Koh1} (with $D=2b$) to each of the projection maps restricted to $\overline{B}_b(p)$ to conclude that $\bigwedge_{j=1}^m\|w_t-P_j(w_t)\|\leq \eta_0$. By convexity, $w_t\in \overline{B}_b(p)$ and the result follows by the assumption on $\eta_0$ and $x$.\hfill$\blacksquare$\\
	
	\noindent By definition of $\tilde{\delta}$, we have $\bigwedge_{j=1}^m\|x-P_j(x)\|\leq \delta(\eta_1)$, and an application of Lemma~\ref{Lem-tech-Koh2} (with $D=2b$) concludes the proof.\qedhere	
\end{proof}

We will also make use of the following two technical lemmas. The first lemma is simply a quantitative version of the fact that any summable sequence of nonnegative real numbers must convergence towards zero. Note that, despite considerations from computability theory excluding the existence of a computable rate of convergence in general, an effective rate to its `\emph{metastable}' formulation is actually trivial to obtain.
\begin{lemma}\label{L1}
	Let $(a_n)\in \ell_{+}^1(\N)$ and consider $B\in\N$ such that $\sum a_n \leq B$. Then,
	\[
	\forall \eps>0\ \forall f:\N\to\N\ \exists n\leq \Psi(B,\eps, f)\ \forall i\in[n; n+f(n)] \left( a_i\leq \eps \right),
	\]
	where $\Psi(B,\eps, f):=\check{f}^{(R)}(0)$ with $\check{f}(p):=p+f(p)+1$ and $R:=\lfloor \frac{B}{\eps}\rfloor$.
\end{lemma}

\begin{proof}
	Let $\eps>0$ and $f:\N\to\N$ be given. Assume towards a contradiction that
	\begin{equation}\label{L1:eq1}
	\forall n\leq \Psi(B,\eps,f)\ \exists i\in[n;n+f(n)] \left(a_i>\eps\right).
	\end{equation}
	By induction we see that for all $r\leq R$,
	\begin{equation}\label{L1:eq2}
		\sum_{i=0}^{\check{f}^{(r)}(0)+f(\check{f}^{(r)}(0))}a_i > (r+1)\eps.
	\end{equation}
	For $r=0$, by \eqref{L1:eq1} there is $i_0\in[0;f(0)]$ such that $a_{i_0}>\eps$, and so $\sum_{i=0}^{f(0)}a_i>\eps$. Assume now that \eqref{L1:eq2} holds for some $r<R$. Then, by \eqref{L1:eq1} with $n=\check{f}^{(r+1)}(0)$ (which is $\leq \Psi(B,\eps,f)$), there is $i_{r+1}\in [\check{f}^{(r+1)}(0); \check{f}^{(r+1)}(0)+f(\check{f}^{(r+1)}(0))]$ such that $a_{i_{r+1}}>\eps$. Thus,
	\begin{align*}
	\sum_{i=0}^{\check{f}^{(r+1)}(0)+f(\check{f}^{(r+1)}(0))} a_i &= \sum_{i=\check{f}^{(r+1)}(0)}^{\check{f}^{(r+1)}(0)+f(\check{f}^{(r+1)}(0))} a_i + \sum_{i=0}^{\check{f}^{(r)}(0)+f(\check{f}^{(r)}(0))}a_i\\
	&> \eps + (r+1)\eps=(r+2)\eps,
	\end{align*}
	concluding the induction. We now have a contradiction from \eqref{L1:eq2} with $r=R$, since it entails that there is $N\in\N$ such that
	\[
	\sum_{i=0}^{N} a_i > (R+1)\eps \geq B.\qedhere
	\]
	
\end{proof}

The next lemma corresponds to a quantitative version of \cite[Lemma~30.6]{BauschkeCombettes(2017)} (which itself is a variant of a technical lemma proved by Dykstra).
\begin{lemma}\label{L2}
Let $(a_n)\in\ell_{+}^2(\N)$ and consider $B \in \N$ such that $\sum a_n^2 \leq B$. For all $n\in\N$, set $s_n:=\sum_{k =0}^{n} a_k$, and let $m\geq 2$ be given. Then, 
\[
\varliminf s_{n}(s_{n}-s_{n-m-1})=0\ \text{with}\ \liminf\text{-rate}\ \phi_B(m,\eps,N):=\left\lfloor e^{\left(\frac{(m+1)B}{\eps}\right)^2}\right\rfloor\cdot (N+1),
\]
i.e.
\begin{equation*}
	\forall \eps>0\ \forall N\in \N\ \exists n\in[N; N+\phi_B(m,k,N)] \left( s_{n}(s_{n}-s_{n-m-1}) \leq \eps\right).
\end{equation*}
\end{lemma}

\begin{proof}
	Let $N \in \N$ and $\eps>0$ be given, and shorten $\phi=\phi_B(m,k,N)$. Using the Cauchy-Schwarz inequality, we have for all $n\in \N$
	\begin{equation}\label{L2:eq1}
	s_{n}\leq \sqrt{n+1}\sqrt{\sum_{k=0}^{n}a_k^2}\leq \sqrt{n+1}\sqrt{B}.
	\end{equation}
	Assume towards a contradiction that for all $n \in [N; N+\phi]$ we have 
	\begin{equation}\label{L2:eq2}
	s_{n}-s_{n-m-1}=\sum_{i=n-m}^{n} a_{i} > \frac{\eps}{\sqrt{B}\sqrt{n+1}}.
	\end{equation}
	Then, again by the Cauchy-Schwarz inequality, for all $n \in [N; N+\phi]$
	\[
	\frac{\eps^2}{B(n+1)}< \left(\sum_{i=n-m}^{n} a_{i}\right)^2 \leq (m+1)\cdot\! \sum_{i=n-m}^{n} a_{i}^2 .
	\]
	Note that, by the integral test, the definition of $\phi$ entails
	\[
	\sum_{n=N}^{N+\phi} \frac{1}{n+1} \geq \log\left(\frac{N+\phi+2}{N+1}\right) \geq \frac{(m+1)^2B^2}{\eps^2},
	\]
	and so we derive the following contradiction
	\begin{align*}
		(m+1)^2B&\leq \sum_{n=N}^{N+\phi} \frac{\eps^2}{B(n+1)}<\sum_{n=N}^{N+\phi}\left((m+1)\cdot\!\sum_{i=n-m}^{n}a_i^2\right)\\
		&=(m+1)\sum_{k=0}^{m}\sum_{i=N}^{N+\phi} a_{i-k}^2 \leq (m+1)^2B.
	\end{align*}
	Hence there must exist some $n\in[N;N+\phi]$ such that \eqref{L2:eq2} fails, and by \eqref{L2:eq1} the result follows.
\end{proof}

\section{Main results}\label{main_section}

For the remaining sections, let $x_0\in X$ be given and consider $(x_n)$ to be the iteration generated by \eqref{dykstra-iter}. We start with some facts that follow easily from the definition of the algorithm.
\begin{lemma}\label{easy-facts}
	For all $n\in\N^*$,
	\begin{enumerate}[label=$(\roman*)$]
		\item\label{eq1} $x_{n-1}-x_n=q_n-q_{n-m}$
		\item\label{eq2} $x_0-x_n=\sum_{k=n-m+1}^{n}q_k$
		\item\label{eq3} $x_n\in C_n$ and $\forall z\in C_n \left( \langle x_n-z, q_n\rangle \geq 0\right)$
		\item\label{eq4} $\langle x_n-x_{n+m}, q_n\rangle \geq 0$.
	\end{enumerate}
\end{lemma}

\begin{proof}
	Let $n\in\N^*$ be given. Fact $(i)$ follows immediately from the definition of $q_n$. Now, from $(i)$ we easily derive $(ii)$. Indeed,
	\begin{align*}
		x_0-x_n&=\sum_{k=1}^{n} x_{k-1}-x_k=\sum_{k=1}^{n} q_{k}-q_{k-m}=\sum_{k=1}^{n}q_k - \sum_{k=-(m-1)}^{n-m}q_k\\
		&=\sum_{k=1}^{n}q_k - \sum_{k=1}^{n-m}q_k=\sum_{k=n-m+1}^{n}q_k.
	\end{align*}
	The definition of $x_n$ entails $(iii)$ using the definition of $q_n$ and the characterization of the metric projection in Proposition~\ref{Kolmogorov}. Finally, point $(iv)$ is an immediate consequence of $(iii)$ as $x_{n+m}\in C_{n+m}=C_{j_{n+m}}=C_{j_n}=C_n$.
\end{proof}

We also have the following useful inequality.
\begin{lemma}\label{easy-facts2}
	For all $n\in\N$, $\sum\limits_{k=n-m+1}^{n}\|q_k\|\leq \sum\limits_{k=0}^{n-1}\|x_k-x_{k+1}\|$.
\end{lemma}

\begin{proof}
We argue by induction on $n\in\N$. The base case $n=0$ is trivial since $q_{-(m-1)}=\cdots=q_0=0$. For the induction step,
\begin{align*}
	\sum_{k=n-m+2}^{n+1}\|q_k\|&=\sum_{k=n-m+1}^{n}\|q_k\| + \|q_{n+1}\|-\|q_{n-m+1}\|\\
	&\overset{{\rm IH}}{\leq} \quad\sum_{k=0}^{n-1}\|x_k-x_{k+1}\| +\|q_{n+1}-q_{n+1-m}\|=\sum_{k=0}^{n}\|x_k-x_{k+1}\|,
\end{align*}
using Lemma~\ref{easy-facts}.\ref{eq1}. This concludes the proof.
\end{proof}

We now prove the main equality used throughout Dykstra's proof.
\begin{lemma}\label{lemma_main_equality}
For all $z\in X$ and $i,n\in\N$ with $i\geq n$,\footnote{Here one considers $x_{-(m-1)}, \cdots, x_{-1}$ arbitrary points in $X$.}
\begin{align}
	\|x_n-z\|^2=& \|x_i-z\|^2 + \sum_{k=n}^{i-1} \left(\|x_k-x_{k+1}\|^2 + 2\langle x_{k-m+1}-x_{k+1},q_{k-m+1}\rangle \right)\nonumber\\
	&\quad+2\sum_{k=i-m+1}^{i}\langle x_k-z,q_k\rangle - 2\sum_{k=n-m+1}^{n}\langle x_k-z,q_k\rangle,\label{main_identity}
\end{align}
and in particular
\begin{equation}\label{main_inequality}
	\|x_i-z\|^2 \leq \|x_n-z\|^2 + 2\sum_{k=n-m+1}^{n}\langle x_k-z,q_k\rangle - 2\sum_{k=i-m+1}^{i}\langle x_k-z,q_k\rangle.
\end{equation}
\end{lemma}

\begin{proof}
 The proof of identity \eqref{main_identity} is by induction on $i$. The base case $i=n$ is trivial. For $i+1$, we first see that
\begin{align*}
	\langle x_{i+1}-z, q_{i+1}-q_{i+1-m}\rangle&=\langle x_{i+1}-z, q_{i+1}\rangle\\
	&\quad -\left( \langle x_{i+1}-x_{i-m+1}, q_{i-m+1}\rangle + \langle x_{i-m+1}-z, q_{i-m+1}\rangle \right)\\
	&=\langle x_{i-m+1}-x_{i+1}, q_{i-m+1}\rangle\\
	&\quad +\langle x_{i+1}-z, q_{i+1}\rangle  -\langle x_{i-m+1}-z, q_{i-m+1}\rangle,
\end{align*}
and so, using Lemma~\ref{easy-facts}.\ref{eq1}
\begin{align*}
	\|x_i-z\|^2&=\langle (x_{i+1}-z) + (x_i-x_{i+1}), (x_{i+1}-z) + (x_i-x_{i+1})\rangle\\
	&=\|x_{i+1}-z\|^2 + \|x_i-x_{i+1}\|^2 + 2\langle x_{i+1}-z, x_i-x_{i+1}\rangle\\
	&=\|x_{i+1}-z\|^2 + \|x_i-x_{i+1}\|^2 + 2\langle x_{i+1}-z, q_{i+1}-q_{i+1-m}\rangle\\
	&=\|x_{i+1}-z\|^2 + \|x_i-x_{i+1}\|^2 + 2\langle x_{i-m+1}-x_{i+1}, q_{i-m+1}\rangle\\
	&\quad +2\langle x_{i+1}-z, q_{i+1}\rangle  -2\langle x_{i-m+1}-z, q_{i-m+1}\rangle.
\end{align*}
The induction step is now easy to verify,
\begin{align*}
	\|x_n-z\|^2&\overset{{\rm IH}}{=} \|x_i-z\|^2 + \sum_{k=n}^{i-1} \left(\|x_k-x_{k+1}\|^2 + 2\langle x_{k-m+1}-x_{k+1},q_{k-m+1}\rangle \right)\\
	&\quad+2\sum_{k=i-m+1}^{i}\langle x_k-z,q_k\rangle - 2\sum_{k=n-m+1}^{n}\langle x_k-z,q_k\rangle\\
	&=\Big(\|x_{i+1}-z\|^2 + \|x_i-x_{i+1}\|^2 + 2\langle x_{i-m+1}-x_{i+1}, q_{i-m+1}\rangle\\
	&\quad +2\langle x_{i+1}-z, q_{i+1}\rangle  -2\langle x_{i-m+1}-z, q_{i-m+1}\rangle\Big)\\
	&\quad +\sum_{k=n}^{i-1} \left(\|x_k-x_{k+1}\|^2 + 2\langle x_{k-m+1}-x_{k+1},q_{k-m+1}\rangle \right)\\
	&\quad+2\sum_{k=i-m+1}^{i}\langle x_k-z,q_k\rangle - 2\sum_{k=n-m+1}^{n}\langle x_k-z,q_k\rangle\\
	&=\|x_{i+1}-z\|^2 + \sum_{k=n}^{i} \left(\|x_k-x_{k+1}\|^2 + 2\langle x_{k-m+1}-x_{k+1},q_{k-m+1}\rangle \right)\\
	&\quad+2\sum_{k=i-m+2}^{i+1}\langle x_k-z,q_k\rangle - 2\sum_{k=n-m+1}^{n}\langle x_k-z,q_k\rangle,
\end{align*}
which concludes the induction and proves \eqref{main_identity}. To verify \eqref{main_inequality}, just note that the terms in the first sum are all nonnegative.
\end{proof}

The assumption of feasibility, $\bigcap_{j=1}^mC_j\neq \emptyset$, now entails that the iteration $(x_n)$ is bounded and $\sum \|x_k-x_{k+1}\|^2<\infty$. For the sequel, we fix some point $p\in\bigcap_{j=1}^mC_j$ and a natural number $b\in\N^*$ such that $b\geq \|x_0-p\|$.
\begin{lemma}
For all $n\in\N$,
\[
\|x_n-p\|\leq b\ \text{and}\
\sum_{k=0}^n\|x_{k}-x_{k+1}\|^2\leq b^2.
\]
\end{lemma}

\begin{proof}
	The result follows immediately from \eqref{main_inequality} with $z=p$ and $n=0$ using Lemma~\ref{easy-facts}.\ref{eq3} and the fact that $\sum_{k=-(m-1)}^{0} \langle x_k-p,q_k\rangle =0$.
\end{proof}

Using Lemma~\ref{L2}, we derive a $\liminf$-rate for the first milestone in the convergence proof of Dykstra's algorithm.
\begin{proposition}\label{lim-inf-rate}
	We have	$\varliminf_n \sum_{k=n-m+1}^{n} |\langle x_k-x_n, q_k\rangle| = 0$,
	and moreover, for all $\eps>0$ and $N\in\N$
	\begin{equation}\label{eq5}
		\exists n\in[N;N+\Phi(b,m,\eps, N)]	\left(\sum_{k=n-m+1}^{n} |\langle x_k-x_n, q_k\rangle| \leq \eps\right),
	\end{equation}
	where $\Phi(b,m,\eps,N):=\phi_{b^2}(m,\eps,N)$, with $\phi$ as defined in Lemma~\ref{L2}.
\end{proposition}

\begin{proof}
	Let $\eps>0$ and $N\in\N$ be given. As we have seen $\sum \|x_k-x_{k+1}\|^2\leq b^2$ and so we can apply Lemma~\ref{L2} (instantiated with $a_n=\|x_n-x_{n+1}\|$ and $B=b^2$) to conclude that there is $n\in[N; N+\Phi(b,m,\eps, N)]$ such that
	\[
	\left(\sum_{k=n-m+1}^{n}\|x_k-x_{k+1}\|\right)\cdot\left(\sum_{k=0}^{n}\|x_k-x_{k+1}\|\right)=(s_n-s_{n-m})s_n\leq \eps.
	\]
	By triangle inequality, for all $k\in[n-m+1;n]$,
	\[
	\|x_k-x_n\|\leq \sum_{\ell=k}^{n-1} \|x_{\ell}-x_{\ell+1}\|\leq \sum_{\ell=n-m+1}^{n-1}\|x_{\ell}-x_{\ell+1}\|,
	\]
	and thus using Cauchy-Schwarz and Lemma~\ref{easy-facts2}, we get
	\begin{align*}
		\sum_{k=n-m+1}^{n} |\langle x_k-x_n, q_k\rangle|&\leq \sum_{k=n-m+1}^{n} \|x_k-x_n\|\cdot\|q_k\|\\
		&\leq \left(\sum_{k=n-m+1}^{n}\|q_k\|\right)\left(\sum_{\ell=n-m+1}^{n-1}\|x_{\ell}-x_{\ell+1}\|\right)\\
		&\leq \left(\sum_{k=0}^{n}\|x_k-x_{k+1}\|\right)\left(\sum_{k=n-m+1}^{n}\|x_k-x_{k+1}\|\right)\leq \eps,
	\end{align*}
	which, in particular, means that $\varliminf_n \sum_{k=n-m+1}^{n} |\langle x_k-x_n, q_k\rangle| = 0$.
\end{proof}

\subsection{Asymptotic regularity}

Here we discuss the asymptotic regularity of the sequence $(x_n)$. Usually, asymptotic regularity \cite{BrowderPetryshyn(1966)} is an intermediate step into establishing the convergence of algorithms approximating fixed points. While convergence says that the iteration approximates a fixed point, asymptotic regularity states that the iteration itself behaves asymptotically as a fixed point. In the previous subsection, we saw that $\sum_{k=0}^{n}\|x_k-x_{k+1}\|^2\leq b^2$. Hence by Lemma~\ref{L1} we immediately have the following result.
\begin{proposition}\label{meta_asymptotic_reg1}
	We have $\lim\|x_n-x_{n+1}\|= 0$, and moreover
	\begin{equation*}
		\forall \eps>0\ \forall f\in\N^{\N}\ \exists n\leq \Psi(b^2, \eps^2, f)\ \forall k\in [n;n+f(n)] \left(\|x_k-x_{k+1}\|\leq \eps\right),
	\end{equation*}
	where $\Psi$ is as defined in Lemma~\ref{L1}.
\end{proposition}

It is now easy to show that $(x_n)$ is asymptotically regular with respect to the individual projection maps.
\begin{proposition}\label{meta_asymptotic_reg2}
	For all $j\in[1;m]$, we have $\lim \|x_n-P_j(x_n)\|=0$, and
	\[
	\forall \eps>0\ \forall f\in\N^{\N}\ \exists n\leq \alpha(b,m, \eps, f)\ \forall k\in [n;n+f(n)]
	\left(\bigwedge_{j=1}^{m} \|x_k-P_j(x_{k})\|\leq \eps\right)
	\]
	where $\alpha(b,m,\eps, f):=\Psi(b^2,\left(\frac{\eps}{m-1}\right)^2,f+m-2)$, with $\Psi$ as defined in Lemma~\ref{L1}.
\end{proposition}

\begin{proof}
	For given $\eps>0$ and $f:\N\to\N$, by Proposition~\ref{meta_asymptotic_reg1} there is $n\leq \alpha(b,m,\eps,f)$ such that
	\begin{equation}\label{eq:AsReg}
		\forall k\in [n;n+f(n)+m-2] \left( \|x_k-x_{k+1}\|\leq \frac{\eps}{m-1}\right).
	\end{equation}
	Consider $k\in [n;n+f(n)]$. By the definition, we have $x_k\in C_{j_k}$ with $j_k:=[k-1]+1$. For all $i\in[0;m-1]$, as $x_{k+i}\in C_{j_k+i}$ from the definition of the projection map $P_{j_k+i}$, we have
	\begin{align*}
		\|x_k-P_{j_k+i}(x_k)\|&\leq \|x_k-x_{k+i}\| \leq \sum_{\ell=k}^{k+i-1} \|x_{\ell}-x_{\ell+1}\|\\
		&\leq \sum_{\ell=k}^{k+m-2} \|x_{\ell}-x_{\ell+1}\| \leq (m-1)\cdot\frac{\eps}{m-1}=\eps,
	\end{align*}
	where in the last step we use the fact that $[k;k+m-2]\subset [n;n+f(n)+m-2]$ and \eqref{eq:AsReg}. The conclusion now follows from observing that for any $k\in\N$ the cyclic definition of $C_{(\cdot)}$ entails $\{P_{j_k+i} : i\in[0;m-1]\}=\{P_1, \cdots, P_m\}$.
\end{proof}

\begin{remark}\label{remark_ass-reg}
	Note that the previous argument is clearly constructive, and the reason we only obtain a rate of metastability is due to having only a metastability statement in Proposition~\ref{meta_asymptotic_reg1}. This is already clear by Proposition~\ref{rate-meta} and the fact that the counterfunction $f$ only appears in the definition of $\alpha$ when the bounding information from Proposition~\ref{meta_asymptotic_reg1} depends on the functional parameter. Indeed, if the conclusion of Proposition~\ref{meta_asymptotic_reg1} would hold with a rate of asymptotic regularity $\psi$,
	\[
	\forall \eps >0\ \forall k \geq \psi(\eps) \left(\|x_k-x_{k+1}\|\leq \eps\right),
	\]
	then the same argument would entail a rate satisfying
	\[
	\forall \eps >0\ \forall k \geq \widetilde{\psi}(\eps) \left(\bigwedge_{j=1}^{m} \|x_k-P_j(x_{k})\|\leq \eps\right),
	\]
	with $\widetilde{\psi}(\eps):=\psi(\eps/(m-1))$.
\end{remark}

\subsection{Metastability}

We show the quantitative version of Theorem~\ref{main} regarding the strong convergence of Dykstra's algorithm. Moreover, the proof of this finitary version bypasses all of the compactness principles used in the original proof. We start with an easy remark regarding some of the data obtained so far.
\begin{remark}\label{monotonicity}
	The function $\alpha$ (Proposition~\ref{meta_asymptotic_reg2}) is monotone in $\eps$,
	\[
	\eps\leq \eps' \to \alpha(b,m,\eps,f)\geq \alpha(b,m,\eps',f).
	\]
	The function $\Phi$ (Proposition~\ref{lim-inf-rate}) is monotone in $N$,
	\[
	N\leq N' \to \Phi(b,m,\eps,N)\leq \Phi(b,m,\eps,N').
	\]
\end{remark}

We have the following result which plays a central role in bypassing the compactness principles used in the original argument.
\begin{proposition}\label{prop-bypass}
	Let $\eps \in (0,1]$ and a function $\Delta:\N\to (0,\infty)$ be given. Then,
	\begin{equation*}
		\begin{split}
			&\exists n\leq \gamma(b,m,\eps,\Delta)\ \exists x\in \overline{B}_b(p)\\
			&\left( \bigwedge_{j=1}^m\|x-P_j(x)\|\leq \Delta(n) \land \|x-x_n\|\leq \eps \land \sum_{k=n-m+1}^{n}\langle x_k-x_n, q_k\rangle\leq \eps \right),
		\end{split}
	\end{equation*}
	where $\gamma(b,m,\eps, \Delta):=\overline{\alpha}(\overline{\beta}) + \Phi_{\eps}\big(\overline{\alpha}(\overline{\beta})\big)$ with
	\[
	\begin{gathered}
	\overline{\beta}:=\beta\left(b,\frac{\eps^2}{2}, \delta\right),\\
	\delta(\eta):=\min\left\{ \frac{\eps^2}{8b\big(\overline{\alpha}(\eta) + \Phi_{\eps}\big(\overline{\alpha}(\eta)\big)\big)}, \widetilde{\Delta}\Big(\overline{\alpha}(\eta) + \Phi_{\eps}\big(\overline{\alpha}(\eta)\big)\Big) \right\},\ \text{for all}\ \eta>0,\\
	\overline{\alpha}(\eta):=\alpha(b,m,\eta, \Phi_{\eps}),\ \text{for all}\ \eta>0,\\
	\Phi_{\eps}\big(N\big):=\Phi\!\left(b,m,\frac{\eps^2}{4}, N\right),\ \text{for all}\ N\in\N,\\
	\widetilde{\Delta}(k):=\min\{\Delta(k') : k'\leq k\},\ \text{for all}\ k\in \N,\\
	\alpha, \beta, \Phi\ \text{are as in Propositions~\ref{meta_asymptotic_reg2}, \ref{quant-projection} and \ref{lim-inf-rate}, respectively}.	
	\end{gathered}
	\]
\end{proposition}

\begin{proof}
	By Proposition~\ref{quant-projection} with $u=x_0$, there are $\eta_0\geq \overline{\beta}$ and $x\in \overline{B}_b(p)$ such that
	\begin{equation}\label{eq12}
		\bigwedge_{j=1}^m\|x-P_j(x)\|\leq \delta(\eta_0)
	\end{equation}
	and
	\begin{equation}\label{eq13}
	\forall y\in \overline{B}_b(p) \left( \bigwedge_{j=1}^m\|y-P_j(y)\|\leq \eta_0 \to \langle x_0-x, y-x\rangle \leq \frac{\eps^2}{2} \right).
	\end{equation}
	Considering Proposition~\ref{meta_asymptotic_reg2} with $\eps=\eta_0$ and $f=\Phi_{\eps}$, we obtain
	\[
	\exists N_0\leq \overline{\alpha}(\eta_0)\ \forall i \in [N_0; N_0+\Phi_{\eps}\big(N_0\big)] \left( \bigwedge_{j=1}^m \|x_i-P_j(x_i)\|\leq \eta_0 \right).
	\]
	Since $(x_n)\subseteq \overline{B}_b(p)$, by \eqref{eq13} we have $\forall i\in [N_0;N_0+\Phi_{\eps}(N_0)] \left( \langle x_0-x, x_i-x\rangle \leq \eps^2/2 \right)$.
	
	On the other hand, from Proposition~\ref{lim-inf-rate} (with $\eps=\frac{\eps^2}{4}$ and $N=N_0$) and the definition of the function $\Phi_{\eps}$, there exists $n_0\in [N_0; N_0+\Phi_{\eps}(N_0)]$ such that
	\[
	\sum_{k=n_0-m+1}^{n_0} \langle x_k-x_{n_0},q_k\rangle \leq \frac{\eps^2}{4}.
	\]
	At this point, we remark that $n_0\leq \gamma(b,m,\eps, \Delta)$. Indeed, by Remark~\ref{monotonicity} and the fact that $\eta_0\geq \overline{\beta}$,
	\begin{align*}
	n_0&\leq N_0+\Phi_{\eps}(N_0)\leq \overline{\alpha}(\eta_0)+\Phi_{\eps}\big(\overline{\alpha}(\eta_0)\big)\leq \overline{\alpha}(\overline{\beta})+\Phi_{\eps}\big(\overline{\alpha}(\overline{\beta})\big)=\gamma(b,m,\eps, \Delta).
	\end{align*}
	The definition of the function $\delta$ and the monotonicity of $\widetilde{\Delta}$, then entail
	\[
	\delta(\eta_0)\leq \widetilde{\Delta}\!\left(\overline{\alpha}(\eta_0) + \Phi_{\eps}\left(\overline{\alpha}(\eta_0)\right)\right)\leq \widetilde{\Delta}\!\left(\overline{\alpha}(\beta) + \Phi_{\eps}\left(\overline{\alpha}(\beta)\right)\right)\leq \Delta(n_0).
	\]
	Hence, the first and the last term of the conjunction in the result hold true \emph{a fortiori}, and it remains to verify that $\|x-x_{n_0}\|\leq \eps$. Note that the definition of $\delta$ also entails
	\[
	\delta(\eta_0)\leq \frac{\eps^2}{8b(\overline{\alpha}(\eta_0)+\Phi_{\eps}(\overline{\alpha}(\eta_0))}\leq \frac{\eps^2}{8b(N_0+\Phi_{\eps}(N_0))}\leq \frac{\eps^2}{8bn_0}.
	\]
	Thus,
	\begin{align*}
		\|x-x_{n_0}\|^2 &=\langle x-x_{n_0}, x-x_0\rangle + \langle x-x_{n_0}, x_0-x_{n_0}\rangle\\
		&\leq \frac{\eps^2}{2} + \langle x-x_{n_0}, x_0-x_{n_0}\rangle\\
		&= \frac{\eps^2}{2} + \sum_{k={n_0}-m+1}^{n_0} \langle x-x_{n_0}, q_k\rangle\qquad \text{by Lemma~\ref{easy-facts}.\ref{eq2}}\\
		&= \frac{\eps^2}{2} + \sum_{k={n_0}-m+1}^{n_0} \langle x_k-x_{n_0}, q_k\rangle + \sum_{k={n_0}-m+1}^{n_0} \langle x-x_k, q_k\rangle\\
		&\leq \frac{\eps^2}{2} + \frac{\eps^2}{4} + \sum_{k={n_0}-m+1}^{n_0} \langle \underbrace{P_k(x)}_{\in C_k}-x_k, q_k\rangle + \sum_{k={n_0}-m+1}^{n_0} \langle x-P_k(x), q_k\rangle\\
		&\leq \frac{3\eps^2}{4} + \sum_{k={n_0}-m+1}^{n_0} \langle x-P_k(x), q_k\rangle\qquad \text{by Lemma~\ref{easy-facts}.\ref{eq3}}\\
		&\leq \frac{3\eps^2}{4} + \sum_{k={n_0}-m+1}^{n_0} \|x-P_k(x)\|\cdot \|q_k\|\\
		&\leq \frac{3\eps^2}{4} + \delta(\eta_0)\sum_{k=0}^{{n_0}-1} \|x_k-x_{k+1}\|\qquad \text{by Lemma~\ref{easy-facts2}}\\
		&\leq \frac{3\eps^2}{4} + \delta(\eta_0)\cdot {n_0}\cdot 2b\leq \frac{3\eps^2}{4} + \frac{\eps^2}{8b{n_0}}\cdot{n_0}\cdot 2b=\eps^2,
	\end{align*}
	which gives $\|x-x_{n_0}\|\leq \eps$ and concludes the proof.	
\end{proof}

We are now ready to prove our central result which corresponds to a quantitative version of Theorem~\ref{main}.
\begin{theorem}\label{main_quant}
	Let $C_1, \cdots, C_m$ be $m\geq 2$ convex subsets of a Hilbert space $X$ such that $\bigcap_{j=1}^m C_j\neq \emptyset$. Let $x_0\in X$ and $b\in\N^*$ be given such that $b\geq \|x_0-p\|$ for some $p\in \bigcap_{j=1}^m C_j$. Then, the sequence $(x_n)$ generated by \eqref{dykstra-iter} is a Cauchy sequence and for all $\eps\in(0,1]$ and $f:\N\to\N$,
	\[
	\exists n\leq \Omega(m,b,\eps, f)\ \forall i,j\in [n;n+f(n)] \left( \|x_i-x_j\|\leq \eps \right),
	\]
	where $\Omega(b,m,\eps,f):=\gamma(b,m,\tilde{\eps}, \Delta_{\eps,f})$ with
	\[
	\begin{gathered}
		\tilde{\eps}:=\frac{\eps^2}{96b},\\
		\Delta_{\eps,f}(k):=\frac{\eps^2}{48b\cdot\max\{k+f(k), 1\}},\ \text{for all}\ k\in\N,\\
		\text{and}\ \gamma\ \text{is as defined in Proposition~\ref{prop-bypass}}.
	\end{gathered}
	\]
\end{theorem}

\begin{proof}
	Let $\eps \in (0,1]$ and a function $f:\N\to\N$ be given. By Proposition~\ref{prop-bypass}, there is $n_0\leq \Omega(b,m,\eps,f)$ and $x\in \overline{B}_b(p)$ such that
	\begin{enumerate}
		\item[$(a)$] $\bigwedge_{j=1}^m \|x-P_j(x)\|\leq \Delta_{\eps,f}(n_0)$,
		\item[$(b)$] $\|x_{n_0}-x\|\leq \tilde{\eps}\leq \frac{\eps}{\sqrt{12}}$,
		\item[$(c)$] $\sum_{k=n_0-m+1}^{n_0} \langle x_k-x_{n_0}, q_k\rangle \leq \tilde{\eps} \leq \frac{\eps^2}{48}$.
	\end{enumerate}
	In order to verify that the result holds for such $n_0$, we consider $i\in [n_0;n_0+f(n_0)]$. We assume that $f(n_0)\geq 1$, and thus $\max\{n_0+f(n_0), 1\}=n_0+f(n_0)$, otherwise the result trivially holds. Since $i\geq n_0$, by \eqref{main_inequality} and using $(b)$, we have
	\begin{align*}
		\|x_i-x\|^2&\leq \|x_{n_0}-x\|^2 +2\sum_{k={n_0}-m+1}^{{n_0}}\langle x_k-x, q_k\rangle - 2\sum_{k=i-m+1}^{i}\langle x_k-x, q_k\rangle\\
		&\leq \frac{\eps^2}{12} + 2\underbrace{\sum_{k={n_0}-m+1}^{n_0}\langle x_k-x, q_k\rangle}_{t_1} + 2\underbrace{\sum_{k=i-m+1}^{i}\langle x-x_k, q_k\rangle}_{t_2}.
	\end{align*}
	Using $(b)$, $(c)$ and Lemma~\ref{easy-facts}.\ref{eq2}, we get
	\begin{align*}
		t_1&=\sum_{k={n_0}-m+1}^{n_0}\langle x_k-x_{n_0}, q_k\rangle + \sum_{k={n_0}-m+1}^{n_0}\langle x_{n_0}-x, q_k\rangle\\
		&\leq \frac{\eps^2}{48} + \langle x_{n_0}-x, \sum_{k={n_0}-m+1}^{n_0}q_k\rangle = \frac{\eps^2}{48} + \langle x_{n_0}-x, x_0-x_{n_0}\rangle\\
		&\leq \frac{\eps^2}{48} + 2b\cdot\|x_{n_0}-x\|\leq \frac{\eps^2}{48} + 2b\cdot\tilde{\eps}=\frac{\eps^2}{48} + 2b\cdot\frac{\eps^2}{96b}=\frac{\eps^2}{24}.
	\end{align*}
	and, using $(a)$ and Kolmogorov's criterium,
	\begin{align*}
		t_2&=\sum_{k=i-m+1}^{i}\langle x-P_k(x), q_k\rangle + \underbrace{\sum_{k=i-m+1}^{i}\langle \underbrace{P_k(x)}_{\in C_k}-x_k, q_k\rangle}_{\leq\, 0}\\
		&\leq \sum_{k=i-m+1}^{i}\langle x-P_k(x), q_k\rangle \leq \sum_{k=i-m+1}^{i} \|x-P_k(x)\|\|q_k\|\\
		&\leq \Delta_{\eps, f}({n_0})\sum_{k=i-m+1}^{i}\|q_k\| \leq \Delta_{\eps, f}({n_0})\cdot\sum_{k=0}^{i-1}\|x_k-x_{k+1}\|\quad \text{by Lemma~\ref{easy-facts2}}\\
		&\leq 2b\cdot i\cdot\Delta_{\eps, f}({n_0})=2b\cdot i\cdot\frac{\eps^2}{48b({n_0}+f({n_0}))}\leq \frac{\eps^2}{24},
	\end{align*}
	using in the last inequality the fact that $i\leq {n_0}+f({n_0})$. Overall, we conclude that
	\[
	\|x_i-x\|^2\leq \frac{\eps^2}{12} + \frac{\eps^2}{12} + \frac{\eps^2}{12}=\frac{\eps^2}{4},
	\]
	and thus $\|x_i-x\|\leq \eps/2$, which entails the result by triangle inequality.
\end{proof}

In contrast with the lack of a full rate of convergence, the reader should note the high uniformity of the rate of metastability obtained. Our function does not depend on specifics of the underlying space nor on any additional geometric properties of the convex sets. The rate only depends on the parameters $m\geq 2$ for the number of sets, and $b\in\N^*$ for a bound on the distance between the initial point and the feasibility set.
\begin{remark}\label{true-finitization}
	Theorem~\ref{main_quant} is a true finitization of Dykstra's convergence result in the sense that, besides only discussing properties for a finite number of terms, it implies back the original statement. Indeed, if the sets are closed, as the sequence $(x_n)$ satisfies the metastability property it is a Cauchy sequence, and by completeness it converges to some point of the space, say $z=\lim x_n$. By Proposition~\ref{meta_asymptotic_reg2} and continuity of the projection maps $P_j$, we conclude that $z$ must be a common fixed point for all the projection, i.e.\ $z\in\bigcap_{j=1}^m C_j$. It only remains to argue that the limit point is actually the feasible point closest to $x_0$. Let $C:=\bigcap_{j=1}^m C_j$ and write $P_C(x_0)$ for the projection of $x_0$ onto the intersection. Consider $\eps>0$ to be arbitrarily given and $N_0\in\N$ such that $\|x_n-z\|\leq \min\{\frac{\eps^2}{8b}, \frac{\eps}{2}\}$ for all $n\geq N_0$ (with $b$ as before). As per Proposition~\ref{lim-inf-rate}, we may consider some $n_0\geq N_0$ such that
	\[
	\sum_{k=n_0-m+1}^{n_0}\langle x_k-x_{n_0},q_k\rangle\leq \frac{\eps^2}{8}.
	\]
	Since $z\in C$, by Proposition~\ref{Kolmogorov}, we also have 
	\[
	\langle P_C(x_0)-x_{n_0}, P_C(x_0)-x_0\rangle\leq \langle z-x_{n_0}, P_C(x_0)-x_0\rangle\leq b\cdot\|z-x_{n_0}\|\leq \frac{\eps^2}{8}.
	\]
	It is now easy to see that
	\begin{align*}
		\|P_C(x_0)-x_{n_0}\|^2& \leq \frac{\eps^2}{8} + \langle P_C(x_0)-x_{n_0}, x_0-x_{n_0}\rangle\\
		&= \frac{\eps^2}{8} + \sum_{k={n_0}-m+1}^{n_0} \langle P_C(x_0)-x_{n_0}, q_k\rangle\qquad \text{by Lemma~\ref{easy-facts}.\ref{eq2}}\\
		&= \frac{\eps^2}{8} + \sum_{k={n_0}-m+1}^{n_0} \langle x_k-x_{n_0}, q_k\rangle + \underbrace{\sum_{k={n_0}-m+1}^{n_0} \langle P_C(x_0)-x_k, q_k\rangle}_{\leq 0,\ \text{by Lemma~\ref{easy-facts}.\ref{eq3}}}\\
		&\leq \frac{\eps^2}{8}+ \frac{\eps^2}{8}= \left(\frac{\eps}{2}\right)^2,
	\end{align*}
	which entails $\|P_C(x_0)-z\|\leq \eps$ and so, as $\eps$ is arbitrary, $z=P_C(x_0)$.
\end{remark}

\section{Rates of convergence and regularity}\label{simp}

In this section, we study the rate of convergence for Dykstra's algorithm under a regularity assumption on the structure of the convex sets. We remark that a regularity assumption on the convex sets $C_1, \cdots, C_m$ is known to allow for rates of convergence already for simpler iterative methods. As explained in \cite[pp~291--292]{KohlenbachLopes-AcedoNicolae(2019)}, a solution to the CFP can be obtained via a Mann-type iteration and, based on the work in \cite{KhanKohlenbach(2014)}, rates of convergence are available under a regularity condition. Moreover, in \cite[pp~288]{KohlenbachLopes-AcedoNicolae(2019)} the authors obtained rates of convergence for \eqref{MAP} even in a general nonlinear setting and an extremely fast rate was given in \cite[Corollary~4.17]{KohlenbachLopes-AcedoNicolae(2019)}. However, these studies were concerned with the interplay between regularity and iterative methods which are Fejér monotone with respect to the feasibility set. In this context, the study of Dykstra's algorithm is of particular interest as the iteration fails to be Fejér monotone and yet it was possible to obtain rates of convergence.

\subsection{The rate of convergence}

Denote $C:=\bigcap_{j=1}^m C_j$ and let $p$ be some point of $X$. We call a function $\mu:\N\times (0,\infty)\to (0, \infty)$ satisfying for all $\eps >0$ and $r\in\N$,
\begin{equation}\label{mod_regularity}\tag{$\star$}
	\forall x\in \overline{B}_r(p) \left( \bigwedge_{j=1}^m\! \|x-P_j(x)\|\leq \mu_r(\eps) \to \exists z\in C\, \|x-z\|\leq \eps \right)
\end{equation}
a \emph{modulus of regularity} for the sets $C_1, \cdots, C_m$ (centred at $p$). Thus, a modulus of regularity let us know how close to the individual sets (i.e.\ $\mu_r(\eps)$-almost $P_j$ fixed point) must a point be so that we are sure that it is close to the intersection set (i.e.\ an $\eps$-almost $P_C$ fixed point). We refer the reader to \cite{KohlenbachLopes-AcedoNicolae(2019)} where this notion was developed and shown to be an effective tool for a unified discussion of several concepts in convex optimization.

\begin{remark}
	Clearly the conclusion of \eqref{mod_regularity} is equivalent to $\|x-P_C(x)\|\leq \eps$, i.e.\ $\dist(x, C)\leq \eps$. Furthermore, the existence of a modulus of regularity centred at $p$ (say $\mu^p$), obviously entails the existence of a modulus of regularity centred at any other $q\in X$ (say $\mu^q$) -- it is easy to verify that for any $q\in X$, $\mu^q_r: \eps\mapsto\mu^p_{r+\lceil\|p-q\|\rceil}(\eps)$ works. The choice of the point $p$ is always clear by the context and so we just write $\mu$ instead of $\mu^p$.
\end{remark}

In the case where a modulus of regularity is available, we can actually give rates of convergence for Dykstra's iteration.
\begin{theorem}\label{rates}
	Consider $x_0\in X$ and a natural number $b\in\N^*$ such that $b\geq \|x_0-p\|$ for some $p\in C$. Let $\mu$ be a function satisfying \eqref{mod_regularity}. Then,
	\[
	\forall \eps >0\ \forall i,j\geq \Theta(b,m,\eps) \left( \|x_i-x_j\|\leq \eps \right),
	\]
	where $\Theta(b,m,\eps):=\alpha(b,m,\mu_b(\tilde{\eps}),\Phi_{\eps})+\Phi_{\eps}(\alpha(b,m,\mu_b(\tilde{\eps}),\Phi_{\eps}))$ with
	\[
	\begin{gathered}
		\tilde{\eps}:=\frac{\eps^2}{32b},\ \Phi_{\eps}(N):=\Phi\left(b,m,\frac{\eps^2}{16},N\right)\ \text{for all}\ N\in\N,\\
		\alpha,\Phi\ \text{are as in Propositions~\ref{meta_asymptotic_reg2} and \ref{lim-inf-rate}, respectively}.
	\end{gathered}
	\]
	In particular, $(x_n)$ converges with rate $\Theta$.
\end{theorem}

\begin{proof}
	By Proposition~\ref{meta_asymptotic_reg2}, there is $N_0\leq \alpha(b,m,\mu_b(\tilde{\eps}),\Phi_{\eps})$ such that
	\[
	\forall n\in [N_0;N_0+\Phi_{\eps}(N_0)] \left(\bigwedge_{j=1}^m \|x_n-P_j(x_n)\|\leq \mu_b(\tilde{\eps})\right).
	\]
	Since $(x_n)\subseteq \overline{B}_b(p)$, by the assumption \eqref{mod_regularity} on $\mu$ it follows that
	\begin{equation}\label{reg-applied}
	\forall n\in [N_0;N_0+\Phi_{\eps}(N_0)]\ \exists z\in C\cap \overline{B}_b(p) \left( \|x_n-z\|\leq \frac{\eps^2}{32b}\right).
	\end{equation}
	Applying Proposition~\ref{lim-inf-rate} with $\eps=\frac{\eps^2}{16}$ and $N=N_0$, we have $n_0\in [N_0; N_0+\Phi_{\eps}(N_0)]$ such that
	\[
	\sum_{k=n_0-m+1}^{n_0} \langle x_k-x_{n_0}, q_k\rangle \leq \frac{\eps^2}{16}.
	\]
	By \eqref{reg-applied}, let $z_0\in C$ be such that $\|z_0-x_{n_0}\|\leq \eps^2/32b$. Thus, for any $i\geq n_0$
	\begin{align*}
		\sum_{k=i-m+1}^{i}\langle x_k-x_{n_0},q_k\rangle &= \underbrace{\sum_{k=i-m+1}^{i} \langle x_k-z_0, q_k\rangle}_{\geq 0,\ \text{by Lemma~\ref{easy-facts}.\ref{eq3}}} + \sum_{k=i-m+1}^{i} \langle z_0-x_{n_0}, q_k\rangle\\
		&\geq \langle z_0-x_{n_0}, \sum_{k=i-m+1}^{i} q_k\rangle\\
		&= \langle z_0-x_{n_0}, x_0-x_i\rangle\qquad \text{by Lemma~\ref{easy-facts}.\ref{eq2}}\\
		&\geq -\|z_0-x_{n_0}\|\cdot\|x_0-x_i\| \geq -\frac{2b\eps^2}{32b}=-\frac{\eps^2}{16}.
	\end{align*}
	Now by \eqref{main_inequality} with $n=n_0$ and $z=x_{n_0}$,
	\[
	\|x_i-x_{n_0}\|^2\leq 2\!\sum_{k=n_0-m+1}^{n_0}\langle x_k-x_{n_0}, q_k\rangle -2\! \sum_{k=i-m+1}^{i}\langle x_k-x_{n_0}, q_k\rangle\leq \frac{\eps^2}{4},
	\]
	which entails that $\|x_i-x_{n_0}\|\leq \eps/2$ and the result follows by triangle inequality.
\end{proof}

In particular, $\Theta$ is also a rate of asymptotic regularity for the sequence $(x_n)$ and, by Remark~\ref{remark_ass-reg}, the function $\Theta':\eps\mapsto \Theta(b,m,\frac{\eps}{m-1})$ is a rate of asymptotic regularity with respect to the individual projections.

We now recall the following class of convex sets in $\R^n$.
\begin{definition}
	A set $C\subseteq \R^n$ is called a basic semi-algebraic convex set in $\R^n$ if there exist $\gamma\geq 1$ convex polynomial functions $g_i$ on $\R^n$ such that
	\[
	C:=\{ x\in \N^n : g_i(x)\leq 0, i\in[1; \gamma] \}.
	\]
\end{definition}

We remark that the class of basic semi-algebraic convex sets is a broad class of convex sets which includes in particular the polyhedral case and the class of convex sets described by convex quadratic functions. It was observed in \cite{KohlenbachLopes-AcedoNicolae(2019)} that the study of Hölderian regularity in \cite{BorweinLiYao(2014)} entails the existence of a modulus of regularity for basic semi-algebraic convex sets with respect to compact sets. As such, we immediately have the following example of application of Theorem~\ref{rates}.

\begin{example}
	Let $C_1, \cdots, C_m\subseteq \R^n$ be basic semi-algebraic sets described by convex polynomials $g_{i,j}$ with degree at most $d\in\N$, and such that $\bigcap_{j=1}^m C_j\neq \emptyset$. Consider some $p\in \R^n$. Then, for any $r\in\N$ there exists $c>0$ such that
	\[
	\mu_r(\eps):=\frac{\left(\eps/c\right)^{\sigma}}{m},\ \text{with}\ \sigma:=\min\left\{ \frac{(2d-1)^n+1}{2}, B(n-1)d^n \right\},
	\]
	where $B(n):=\Big(\overset{n}{\lfloor n/2 \rfloor}\Big)$ is the central binomial coefficient with respect to $n$, is a modulus of regularity for $C_1, \cdots, C_m$ centred at $p$\footnote{Which corresponds to a modulus of regularity for $C_1, \cdots, C_m$ with respect to the compact set $\overline{B}_r(p)\subseteq \R^n$, in the terminology used in \cite{KohlenbachLopes-AcedoNicolae(2019)}.}. Hence, by Theorem~\ref{rates} one has a uniform rate of convergence for Dykstra's cyclic projections algorithm for basic semi-algebraic convex sets in $\R^n$.
\end{example}

\subsection{Regularity}

We now argue that a modulus of regularity is a necessary condition for the existence of uniform convergence rates.

\begin{proposition}\label{kappa}
	Let $x_0\in X$ and $b\in\N^*$ be such that $b\geq \|x_0-p\|$ for some $p\in\bigcap_{j=1}C_j$. Consider $(x_n)$ the iteration generated by \eqref{dykstra-iter} with initial point $x_0$. Then,
	\[
	\forall \eps>0\ \forall n\in\N \left( \bigwedge_{j=1}^m\|x_0-P_j(x_0)\|\leq \frac{\eps^2}{4bn} \to \|x_n-x_0\|\leq \eps \right).
	\]
\end{proposition}

\begin{proof}
	Let $\eps>0$ and $n\in\N$ be given. Assuming the premise of the implication, by \eqref{main_inequality} (with $i=n$, $n=0$ and $z=x_0$), we have
	\begin{align*}
		\|x_n-x_0\|^2&\leq 2\underbrace{\sum_{k=-(m-1)}^{0}\langle x_k-x_0,q_k\rangle}_{=0} + 2\sum_{k=n-m+1}^{n}\langle x_0-x_k,q_k\rangle\\
		&= \sum_{k=n-m+1}^{n}\langle x_0-P_k(x_0),q_k\rangle + 2\sum_{k=n-m+1}^{n}\underbrace{\langle \underbrace{P_k(x_0)}_{\in C_k}-x_k,q_k\rangle}_{\leq 0\ \text{by Lemma~\ref{easy-facts}.\ref{eq3}}}\\
		&\leq 2\sum_{k=n-m+1}^{n}\|x_0-P_k(x_0)\|\cdot\|q_k\|\leq \frac{\eps^2}{2bn}\sum_{k=n-m+1}^{n}\|q_k\|\\
		&\leq \frac{\eps^2}{2bn}\sum_{k=0}^{n-1}\|x_k-x_{k+1}\|\leq\eps^2\qquad \text{using Lemma~\ref{easy-facts2}},
	\end{align*}
	which entails $\|x_n-x_0\|\leq \eps$ and concludes the proof.
\end{proof}

Since the natural proof that the scheme \eqref{dykstra-iter} satisfies
\[
	x_0\in \bigcap_{j=1}^mC_j \to \forall n\in\N \left(x_n=x_0\right)
\]
doesn't require the knowledge that the functions $P_j$ are the metric projections (and already holds if for example they are only required to be nonexpansive maps, as it suffices to guarantee their extensionality), logical considerations make it clear that there must exist a bound which does not depend on the additional constant $b\in\N$. With such perspective, we give an alternative version of Proposition~\ref{kappa}.
\begin{proposition}\label{kappa2}
	Consider $(x_n)$ to be the iteration generated by \eqref{dykstra-iter} with some initial point $x_0\in X$. Then,
	\[
	\forall \eps>0\ \forall n\in\N^* \left( \bigwedge_{j=1}^m\|x_0-P_j(x_0)\|\leq \frac{\eps}{5^{n-1}} \to \|x_n-x_0\|\leq \eps \right).
	\]
\end{proposition}

\begin{proof}
	By induction we show the stronger assertion that for all $n\in\N^*$
	\[
	\forall \eps>0 \left( \bigwedge_{j=1}^m\|x_0-P_j(x_0)\|\leq \frac{\eps}{5^{n-1}} \to \forall n'\in [1;n] \left( \|x_{n'}-x_0\|\leq \eps \land \|q_{n'-m}\|\leq \eps \right)\right).
	\]
	For $n=1$, we have $q_{1-m}=0$ and
	\[
	\|x_1-x_0\|=\|P_1(x_0)-x_0\|\leq \frac{\eps}{5^0}=\eps.
	\]
	Assuming that the claim holds for some $n\in \N^*$, suppose that
	\[
	\bigwedge_{j=1}^m \|x_0-P_j(x_0)\|\leq \frac{\eps}{5^n}.
	\]
	By the induction hypothesis, we have
	\begin{equation}\label{IH_kappa2}
	\forall n'\in [1;n] \left( \|x_{n'}-x_0\|\leq \frac{\eps}{5} \land \|q_{n'-m}\|\leq \frac{\eps}{5} \right),
	\end{equation}
	and, in particular, we just need to verify the consequent for $n'=n+1$. Let us first focus on $q_{n+1-m}$. If $n<m$, then $q_{n+1-m}=0$ and so we assume $n\geq m$. We have,
	\begin{equation*}
		\|q_{n+1-m}\|=\|x_{n-m}+q_{n+1-2m}-x_{n+1-m}\|\leq \|x_{n-m}-x_{n+1-m}\|+\|q_{n+1-2m}\|
	\end{equation*}
	Since $n+1-m\in [1;n]$, by \eqref{IH_kappa2} we have $\|q_{n+1-2m}\|\leq \eps/5$. On the other hand, we have $\|x_{n-m}-x_{n+1-m}\|\leq 2\eps/5$. Indeed, if $n=m$ then
	\[
	\|x_{n-m}-x_{n+1-m}\|=\|x_0-x_1\|=\|x_0-P_1(x_0)\|\leq \frac{\eps}{5^n}\leq \frac{2\eps}{5}.
	\]
	If $n>m$, then $n-m$, $n+1-m\in [1;n]$ and by \eqref{IH_kappa2}
	\[
	\|x_{n-m}-x_{n+1-m}\|\leq \|x_{n-m}-x_0\|+\|x_{n+1-m}-x_0\|\leq \frac{2\eps}{5}.
	\]
	Overall, we conclude $\|q_{n+1-m}\|\leq 3\eps/5$, which in particular gives the second conjunct. It is now easy to verify that
	\begin{align*}
		\|x_{n+1}-x_0\|&\leq \|P_{n+1}(x_n+q_{n+1-m})-P_{n+1}(x_0)\| + \|P_{n+1}(x_0)-x_0\|\\
		&\leq \|x_n-x_0\|+\|q_{n+1-m}\|+\|P_{n+1}(x_0)-x_0\|\\
		&\leq \frac{\eps}{5} + \frac{3\eps}{5} + \frac{\eps}{5^n} \leq \eps,
	\end{align*}
	concluding the proof.\qedhere	
\end{proof}

We now have the following result stating that the existence of rates of convergence that are uniform for initial points in $\overline{B}_b(p)$, will entail the existence of a modulus of regularity for the convex sets.
\begin{proposition}\label{rate_to_regul}
	Let $p\in \bigcap_{j=1}^m C_j=:C$. For any $b\in\N$, assume the existence of a common rate of convergence towards the limit $P_C(x_0)$ for any iteration generated by \eqref{dykstra-iter} with initial point $x_0\in\overline{B}_b(p)$, i.e.
	\[
	\|x_0-p\|\leq b\ \to\ \forall \eps>0\ \forall n\geq\rho(b,\eps) \left( \|x_n-P_C(x_0)\|\leq \eps\right).
	\]
	Then, the function 
	\[
	\mu(b,\eps):=\max\left\{\frac{\eps^2}{16b\cdot\rho(b,\eps/2)}, \frac{\eps}{2\cdot5^{\rho(b, \eps/2)}}\right\}
	\]
	is a modulus of regularity for the sets $C_1, \cdots, C_m$ centred at $p$.
\end{proposition}

\begin{proof}
	Consider $\eps>0$, $b\in\N$ and $x\in \overline{B}_b(p)$, and assume that
	\[
	\bigwedge_{j=1}^m\|x-P_j(x)\|\leq \max\left\{\frac{\eps^2}{16b\cdot\rho(b,\eps/2)}, \frac{\eps}{2\cdot5^{\rho(b, \eps/2)}}\right\}.
	\]
	Let $(x_n)$ be the iteration generated by \eqref{dykstra-iter} with initial point $x_0:=x$. Then by both Proposition~\ref{kappa} and \ref{kappa2}, $\|x_{\rho(b,\eps/2)}-x\|\leq \frac{\eps}{2}$. On the other hand, by the assumption on $\rho$, we have $\|x_{\rho(b,\eps/2)}-P_C(x)\|\leq \eps/2$. Hence,
	\[
	\|x-P_C(x)\|\leq \|x-x_{\rho(b,\eps/2)}\|+\|x_{\rho(b,\eps/2)}-P_C(x)\|\leq \eps.\qedhere
	\]
\end{proof}

Note that the requirement $p\in C$ is tied with Proposition~\ref{kappa} and can we can take $p$ an arbitrary point in $X$ with suitable changes. The key idea of the previous argument is that a rate of convergence for an iteration which remains constant whenever the initial point is already in the target set, will entail the existence of a modulus of regularity. An analogous argument was used in \cite[Proposition~4.4]{KohlenbachLopes-AcedoNicolae(2019)} for the case of the Picard iteration in a metric setting. Indeed, this reasoning can be stated in a general framework, but we refrain from doing it here and direct the reader to \cite{Pinto(2023)}.

If a rate of convergence is available but it is sensitive to the initial point, then we obtain a weaker result (with unclear usefulness).
\begin{proposition}
	For any $x_0\in X$, assume the existence of rate of convergence towards $P_C(x_0)$ for the iteration generated by \eqref{dykstra-iter} with initial point $x_0$, i.e.
	\[
	\forall \eps>0\ \forall n\geq\rho(x_0,\eps) \left( \|x_n-P_C(x_0)\|\leq \eps\right).
	\]
	Then,
	\[
	\forall \eps >0\ \forall x\in X \left( \bigwedge_{j=1}^m\|x-P_j(x)\|\leq \frac{\eps}{2\cdot5^{\rho(x, \eps/2)}} \to \|x-P_C(x)\|\leq \eps \right).
	\]
\end{proposition}

For the particular case of the intersection of two half-spaces, Duetsch and Hundal~\cite{DeutschHundal(1994)} obtained rate of convergence for Dykstra's algorithm which are uniform on the choice of the initial point depending only on a bound to its distance to the intersection set. By Propositions~\ref{rate_to_regul}, such situation entails a modulus of regularity for the two half-spaces. In full generality, but provided there exists a modulus of regularity, Theorem~\ref{rates} guarantee the existence of uniform rates of convergence. This still leaves open the possibility that no modulus of regularity exist and yet rates of convergence are available. Such rates would necessarily be sensitive to the initial point of the iteration -- such is the case with the rates of convergence obtained by Deutsch and Hundal for the general polyhedral case. In contrast, we obtained rates of metastability in full generality which are uniform in all the parameters of the convex feasibility problem.

\section{Final Remarks}\label{final}

This quantitative study analyzes the proof of strong convergence of Dykstra's cyclic projection algorithm. Although the original proof relies on several strong mathematical principles, in the end we obtain simple computable metastability rates (primitive recursive in $f$ in the sense of Kleene) which are highly uniform in the parameters of the convex feasibility problem. Indeed, our rates only require information on the number of convex sets $m$, and an upper bound $b$ on the distance between the initial point and the feasibility set. Furthermore, under an regularity assumption, we adapt the argument to actually derive uniform rates of convergence towards the feasible point closest to the starting term of the iteration. We show that the regularity condition is actually necessary for the existence of such uniform rates. The regularity assumption comes in the form of a modulus of regularity $\mu$ which (informally) guarantees that the point is $\eps$-close to the intersection whenever it is $\mu(\eps)$-close to all the individual convex sets. In the general case, the finitary version follows through by the crucial observation that the role of the weak limit can actually be replaced by that of a weak version of the projection of $x_0$ onto the intersection set. We show that our main result (Theorem~\ref{main_quant}) is a true finitary version of Theorem~\ref{main} in the sense that it only regards a finite number of iteration terms and the full statement is fully recovered in an elementary way from the quantitative version (cf.\ Remark~\ref{true-finitization}).

This kind of argument is in line with the macro developed in \cite{FerreiraLeusteanPinto(2019)}. The ability to establish the Cauchy property of the iteration without the use of sequential weak compactness is of paramount relevance as it ensures that the final quantitative bound information will be of a simple nature (namely, it can be described without the need of Spector's bar recursive functional \cite{Spector(1962)}). This technique has been applied several times in proof mining (e.g.\ in \cite{Pinto(2021),DinisPinto(2021)i,DinisPinto(2021)ii}). Moreover, even if one is not concerned with quantitative information, a simpler proof which bypasses complex comprehension principles (in this case the arithmetical comprehension required to justify weak compactness) allows for easier generalizations of the original result (see e.g.\ the recent \cite{DinisPinto(2023)ii} where a quantitative approach allowed to establish a fully new result in a geodesic setting in which weak compactness arguments, common in Hilbert spaces, are harder to employ).

Previous eliminations of weak compactness principles were applied to Halpern-type iterations and to convergence proofs following a similar common structure (reminiscent of Wittmann's argument in \cite{Wittmann(1992)}). However, Dykstra's algorithm doesn't appear to have any connection with the Halpern iteration and the proof follows a completely different argument. Thus, it was not a priori known if it would be possible to bypass the compactness arguments crucial in the original proof. Furthermore, regarding the discussion under a regularity assumption, as is explained in \cite{KohlenbachLopes-AcedoNicolae(2019)}, it is known that for Fejér monotone iterations a modulus of regularity allows one to obtain rates of convergence. Note that in this case however, Dykstra's method fails to be Fejér monotone, and still it was possible to extract uniform rates of convergence.\\


\noindent
{\bf Acknowledgment:} This work benefited from discussions with Ulrich Kohlenbach and Nicholas Pischke. The author was supported by the DFG Project KO 1737/6-2.

\bibliographystyle{plain}
\bibliography{References}
\end{document}